\documentclass[11pt]{amsart}
\usepackage[english]{babel}
\usepackage{amssymb,amsmath,amsthm}
\usepackage[latin2]{inputenc}
\usepackage{amsfonts, verbatim,  amscd}

\reversemarginpar
\newtheorem{theorem}{Theorem}
\newtheorem{proposition}[theorem]{Proposition}

\newtheorem{coro}[theorem]{Corollary}
\def\acknowledgment{\par\addvspace{17pt}\small\rmfamily
\trivlist\if!\ackname!\item[]\else
\item[\hskip\labelsep
{\bfseries\ackname}]\fi}

\def\C{\mathbb{C}}

\def\D{\mathbb{D}}

\newcommand{\cO}{\mathcal{O}}
\newcommand{\clD}{{\overline{\,\D}}}

\begin{document}
\title[On Poletsky theory of discs in compact manifolds]{On Poletsky theory of discs in compact manifolds} 
\author{Uro\v s Kuzman }

%    General info
%
%\subjclass[2000]{32Q60,32Q65,32V40,35J56}
%\date{September 6, 2010}
%\keywords{almost complex manifolds, J-holomorphic disc, maximal totally real manifolds, Riemann-Hilbert problem for Pascali systems}
%\keywords{almost complex manifolds, J-holomorphic discs, Disc functionals, J-plurisubharmonic functions}

\begin{abstract}
We provide a direct construction of Poletsky discs via local arc approximation and a Runge-type theorem by A.\ Gournay \cite{Gournay}.
\end{abstract} 
\maketitle
Let $\D=\{\zeta\in \C\colon |\zeta|< 1\}$ denote the open unit disc. Given a smooth (almost) complex manifold $M$ and $p\in M$, 
we denote by $\cO(\clD,M,p)$ 
the set of smooth maps $u\colon \clD \to M$ that are (pseudo)holomorphic in some neighborhood of $\clD$ and satisfy $u(0)=p$. Given $p\in M$, $\epsilon>0$ and an open set $U\subset M$ we call an element of $u\in\cO(\clD,M,p)$ a \textit{Poletsky disc} (associated to $p$, $\epsilon$ and $U$) if most of its boundary lies in $U$, t.i., there exists an exceptional set $E\subset[0,2\pi)$ of Lebesgue measure $|E|<\epsilon$ and such that $u(e^{it})\in U$ for $t\notin E$. Such discs were used by E.\ Poletsky \cite{Poletsky1991} in order to characterize the polynomial hull for compact sets in $\mathbb{C}^n$. Similarly, they can describe the projective hull of a compact set in complex projective spaces \cite{BDD, MAG}.  

All of the above mentioned characterizations are based on the following explicit formula for the largest plurisubharmonic minorant of a given upper-semicontinuous function $f$ on $M$:
\begin{eqnarray}\label{Poletsky}\hat{f}(p)=\inf\left\{\int_0^{2\pi}f\circ u(e^{it})\frac{dt}{2\pi}:\;\; u\in\cO(\clD,M,p) \right\}.\end{eqnarray}
The latter was proved to be valid on any complex manifold by J.\ P.\ Rosay (see \cite{LS} for related results), who also observed that if $M$ admits no non-constant bounded plurisubharmonic function, there exists a Poletsky disc for any $p\in M$, $\epsilon>0$ and an open set $U\subset M$ \cite[Corollary 0.2]{Rosay1}. Indeed, in this case the minorant of the negative indicator function $f=-\chi_U$ equals to $\hat{f}\equiv -1$, hence the existence of the desired discs follows directly from the definition of the infimum in (1). 

In this paper we present a new, direct proof of this corollary valid for a certain class of  manifolds admitting a Runge-type approximation provided by A. Gournay \cite{Gournay}. In particular, we give a partial answer to Rosay's question raised in \cite[Section $5$]{Rosay1}: given a compact complex manifold can a Poletsky disc be provided without using $(1)$? Moreover, our theorem includes also some examples of almost complex manifolds. 
\begin{theorem}
Let $M$ be a smooth, connected compact manifold equipped with a regular almost complex structure and admitting a doubly tangent property. Given a point $p\in M$, a positive constant $\epsilon>0$ and any open set $U\subset M$, there exist a disc $u\in\cO(\clD,M,p)$ and a set $E\subset[0,2\pi)$ such that $|E|<\epsilon$ and $u(e^{it})\in U$ for $t\notin E$.
\end{theorem}
\noindent The assumptions in the above theorem are rather technical and should be read as 'such that Gournay's approximation result applies' (we further discuss them in \S 2). However, they are fulfilled for a wide class of compact complex manifolds including complex projective spaces and Grassmannians. Moreover, as mentioned, they are valid for some manifolds equipped with a non-integrable almost complex structure $J$ (e.g. $\mathbb{C}P^n$ with $J$ tamed by the standard symplectic form) \cite[p.313]{Gournay}. Note that for latter the Poletsky-Rosay formula (1) is proved only in a low dimensional case \cite{KUZMAN3}. Hence for $\dim_{\mathbb{R}}M\geq 6$ and a non-integrable $J$, the theorem is new.

Finally, let us remark that in \S 1 we present another original statement that will be needed in the proof of the main theorem: based on \cite{BK} we provide a Mergelyan-type result for maps defined on smooth arcs (Theorem 5).         

\section{The local arc approximation}
Let $M$ be a smooth real manifold of even dimension. A $\left(1,1\right)$-tensor field $J\colon TM\to TM$ satisfying $J^{2}=-Id$ is called {\it almost complex structure}. A differentiable map $u:\left(M',J'\right) \longrightarrow \left(M,J\right)$ between two almost complex manifolds 
is {\it $\left(J',J\right)$-holomorphic} if for every $p \in M'$ we have
\begin{eqnarray}\label{J-holomorphicity condition}
J\left(u\left(p\right)\right)\circ d_{p}u=d_{p}u\circ J'\left(p\right),
\end{eqnarray}
We deal with two simplest cases, {\it $J$-holomorphic discs} $u\colon \mathbb{D}\to M$ and {\it $J$-holomorphic spheres} $u\colon\mathbb{C}P^1\to M$. 

We denote by $J_{st}$ the standard integrable structure on $\C^{n}$ for any $n\in\mathbb{N}$. In local coordinates $z\in\mathbb{R}^{2n}$ an almost complex structure $J$ is represented by a $\mathbb{R}$-linear operator satisfying $J(z)^2=-I$, hence (2) equals to
\begin{eqnarray}\label{J-hol}u_y=J(u)u_x.\end{eqnarray}
Further, if $J+J_{st}$ is invertible along $u$ we have
\begin{eqnarray}\label{A}\mathcal{F}(u)=u_{\bar{\zeta}}+A(u)\overline{u_{\zeta}}=0,\end{eqnarray}
where $\zeta=x+iy$ and $A(z)(v)=(J_{st}+J(z))^{-1}(J(z)-J_{st})(\bar{v})$
is a complex linear endomorphism for every $z\in\C^n$. We call $A$ the \emph{complex matrix of} $J$ and denote by $\mathcal{J}$ the set of all smooth structures on $\mathbb{R}^{2n}$ satisfying the condition $\det (J+J_{st})\neq 0$.  

In \cite{BK} the approximation theory was developed for the operator $\mathcal{F}$ defined as in (\ref{A}) and evaluated in functions admitting a Sobolev weak derivative. In particular, given $\varphi\in W^{1,p}(\mathbb{D})$, $p>2$, a bounded right inverse $Q_\varphi$ was constructed for the derivative $d_\varphi\mathcal{F}$ and the following version of the Implicit Function Theorem was applied. 
\begin{theorem}[Implicit Function Theorem]
\label{theoimpl}
Let $X$ and $Y$ be two Banach spaces and consider a  
map $\mathcal{F}: U\subset X \to Y$ of class $\mathcal{C}^1$ defined on an open set $U\subset X$. 
Let $x_0\in U$. Assume that the differential $d_{x_0}\mathcal{F}$ admits a bounded right inverse, denoted by $Q_{x_0}$. Fix $c_0>0$ such that $\left\|Q\right\|\leq c_0$ and $\eta>0$ such that if $\left\|x-x_0\right\|<\eta$ then $x\in U$ and $$\left\|d_x\mathcal{F}-d_{x_0}\mathcal{F}\right\|\leq \frac{1}{2c_0}.$$
Then if $\displaystyle \left\|\mathcal{F}(x_0)\right\|<\frac{\eta}{4c_0}$ there exists $x\in U$ such that $\mathcal{F}(x)=0$ and 
$$\left\|x-x_0\right\|\leq2c_0\left\|\mathcal{F}(x_0)\right\|.$$
\end{theorem}

In this paper we use the same approach to develop a similar statement for the operator $\mathcal{F}\colon\mathcal{C}^{1,\alpha}(\overline{\mathbb{D}})\to \mathcal{C}^{0,\alpha}(\overline{\mathbb{D}})$. That is, we prove an analogue of \cite[Theorem 5]{BK} valid for H\"older spaces (we omit the existence part for $Q_\varphi$ since the proof is the same as in \cite[Theorem 2]{BK} and \cite{ST1}). Note that these additional regularity conditions are needed since in the present paper we apply the statement to shrinking neighborhoods of an arc in order to obtain a $\mathcal{C}^0$-approximation. Since the diameter of such sets admits no lower bound, the $W^{1,p}$-result along with the Sobolev embedding theorem does not suffice.          

\begin{theorem}\label{theoap}
Let $0<\alpha<1$. Let $J\in \mathcal{J}$ and let $A$ be its complex matrix. We define $\mathcal{F}:\mathcal{C}^{1,\alpha}(\overline{\mathbb{D}}) \to \mathcal{C}^{0,\alpha}(\overline{\mathbb{D}})$ to be the operator given by 
$$\mathcal{F}(u)=\displaystyle u_{\bar{\zeta}}+A(u)\overline{u_{\zeta}}.$$
For every $c_0>0$, there exists $\delta>0$ such that for any $\varphi \in \mathcal{C}^{1,\alpha}(\overline{\mathbb{D}})$ satisfying 
$$\|\varphi\|_{\mathcal{C}^{1,\alpha}(\overline{\mathbb{D}})}\leq c_0,\;\; \|Q_\varphi\|\leq c_0,\;\;  \left\|\mathcal{F}(\varphi)\right\|_{\mathcal{C}^{0,\alpha}(\overline{\mathbb{D}})}<\delta,$$
there exists a $J$-holomorphic disc $u\in \mathcal{C}^{1,\alpha}(\overline{\mathbb{D}})$  such that 
$$\left\|u-\varphi\right\|_{\mathcal{C}^{1,\alpha}(\overline{\mathbb{D}})}\leq 2c_0\left\|\mathcal{F}(\varphi)\right\|_{\mathcal{C}^{0,\alpha}(\overline{\mathbb{D}})}.$$ 
\end{theorem}

\begin{proof}
The key step is to prove that the derivative of $\mathcal{F}$ is locally Lipschitz under present assumptions. That is, there exists $c>0$ such that 
$$\left\|d_{\tilde{\varphi}}\mathcal{F}-d_{\varphi}\mathcal{F}\right\|\leq c\|\tilde{\varphi}-\varphi\|_{\mathcal{C}^{1,\alpha}(\overline{\mathbb{D}})}$$ for any $\tilde{\varphi} \in \mathcal{C}^{1,\alpha}(\overline{\mathbb{D}})$ in a $\mathcal{C}^{1,\alpha}$-neighborhood of $\varphi$, say $\displaystyle \|\tilde{\varphi}-\varphi\|_{\mathcal{C}^{1,\alpha}(\overline{\mathbb{D}})}< 1$. If such a statement is valid then one can simply set 
$$\eta=\min\left\{1,\frac{1}{2cc_0}\right\}\;\;\textrm{ and }\;\; \displaystyle \delta=\frac{\eta}{4c_0}$$ and apply the Theorem $2$ to $X= \mathcal{C}^{1,\alpha}(\overline{\mathbb{D}})$, $Y=\mathcal{C}^{0,\alpha}(\overline{\mathbb{D}})$ and $x_0=\varphi.$

Let $h \in \mathcal{C}^{1,\alpha}(\overline{\mathbb{D}})$ and let $\left\|\varphi\right\|_{\mathcal{C}^{1,\alpha}(\overline{\mathbb{D}})}<c_0$. We need to prove that 
\begin{equation}\label{eqlip}
\displaystyle \left\|d_{\tilde{\varphi}}\mathcal{F}(h)-d_{\varphi}\mathcal{F}(h)\right\|_{\mathcal{C}^{0,\alpha}(\overline{\mathbb{D}})}\leq c\|\tilde{\varphi}-\varphi\|_{\mathcal{C}^{1,\alpha}(\overline{\mathbb{D}})}\|h\|_{\mathcal{C}^{1,\alpha}(\overline{\mathbb{D}})}.
\end{equation}
Note that    
$$ 
d_{\varphi}\mathcal{F}(h) =  h_{\bar \zeta}+A(\varphi)\overline{h_{\zeta}}+d_\varphi A ( h)  \ \overline{\varphi_\zeta},
$$
where $\displaystyle d_\varphi A  (h)=\sum_{j=1}^n\frac{\partial A}{\partial z_j}(\varphi)h_j+\frac{\partial A}{\partial \bar{z}_j}(\varphi)\bar{h}_j$.
We write $$d_{\tilde{\varphi}}\mathcal{F}(h)-d_{\varphi}\mathcal{F}(h)=I+II+III,$$ where 
$$
\left\{
\begin{array}{lll} 
I&= &\left(A(\tilde{\varphi})-A(\varphi)\right)\overline{h_{\zeta}},\\
II&=&(d_{\tilde{\varphi}} A-d_\varphi A)(h)\overline{\tilde{\varphi}_\zeta},\\
III &= &d_{\varphi} A  (h)\left(\overline{\tilde{\varphi}_\zeta-\varphi_\zeta}\right).
\end{array}
\right.$$                                                                                                                                 
Let us estimate each of the three parts. 

Since $\mathbb{D}$ is convex and bounded the following embeddings are compact
$$\mathcal{C}^{1,\alpha}(\overline{\mathbb{D}})\subset\mathcal{C}^{1}(\overline{\mathbb{D}})\subset\mathcal{C}^{0,\alpha}(\overline{\mathbb{D}}).$$
Therefore, one can bound the $\mathcal{C}^{0,\alpha}$-norm of the entries in $A(\tilde{\varphi})-A(\varphi)$ by their  $\mathcal{C}^1$-norm. This implies existence of a constant $c_1>0$ depending on $\left\|\tilde{\varphi}\right\|_{\mathcal{C}^{1}(\overline{\mathbb{D}})}<1+c_0$ such that
$$\left\|I\right\|_{\mathcal{C}^{0,\alpha}(\overline{\mathbb{D}})}\leq c_1 \|\tilde{\varphi}-\varphi\|_{\mathcal{C}^{1,\alpha}(\overline{\mathbb{D}})}\|h\|_{\mathcal{C}^{1,\alpha}(\overline{\mathbb{D}})}.$$
By a similar argument for $\frac{\partial A}{\partial z_j}(\tilde{\varphi})-\frac{\partial A}{\partial z_j}(\varphi)$ and $\frac{\partial A}{\partial \bar{z}_j}(\tilde{\varphi})-\frac{\partial A}{\partial \bar{z}_j}(\varphi)$ we get
$$\left\|II\right\|_{\mathcal{C}^{0,\alpha}(\overline{\mathbb{D}}}
%\leq c_2 \left\|\varphi\right\|_{\mathcal{C}^{1,\alpha}(\overline{\Delta})}  \|\tilde{\varphi}-\varphi\|_{\mathcal{C}^{1,\alpha}(\overline{\Delta})}\|h\|_{\mathcal{C}^{1,\alpha}(\overline{\Delta})}
\leq c_0c_2  \|\tilde{\varphi}-\varphi\|_{\mathcal{C}^{1,\alpha}(\overline{\mathbb{D}})}\|h\|_{\mathcal{C}^{1,\alpha}(\overline{\mathbb{D}})}.$$
Finally, we have
$$\left\|III\right\|_{\mathcal{C}^{0,\alpha}(\overline{\mathbb{D}})}\leq c_3 \|\tilde{\varphi}-\varphi\|_{\mathcal{C}^{1,\alpha}(\overline{\mathbb{D}})}\|h\|_{\mathcal{C}^{1,\alpha}(\overline{\mathbb{D}})},$$
where $c_3>0$ depends on the $\mathcal{C}^{0,\alpha}$-norm evaluated on the coefficients of $\frac{\partial A}{\partial z_j}(\varphi)$ and $\frac{\partial A}{\partial \bar{z}_j}(\varphi).$ The latter may be bounded by a constant times $c_0$. 
\end{proof}

As stated above we will apply this approximation result to a shrinking family of sets in $\mathbb{D}$. Let $\left\{\Omega_m\right\}_{m\in\mathbb{N}}$ be such a family. Throughout the rest of this section we use the following convention: Given $\varphi\in\mathcal{C}^{1,\alpha}(\overline{\mathbb{D}})$ we denote by $\varphi_m$ its restriction to the set $\Omega_m$ and by $\mathcal{F}_m$ the corresponding operator defined as in (\ref{A}) but mapping from $\mathcal{C}^{1,\alpha}(\overline{\Omega}_m)$ to $\mathcal{C}^{0,\alpha}(\overline{\Omega}_m)$. 

\begin{proposition}
Let $\left\{\Omega_m\right\}_{m\in\mathbb{N}}$ be a shrinking family of sets in $\mathbb{D}$ and let  $\varphi\in\mathcal{C}^{1,\alpha}(\overline{\mathbb{D}})$ be a map satisfying the limit condition
$$\lim_{m\to\infty}\left\|\mathcal{F}_{m}(\varphi_{m})\right\|_{\mathcal{C}^{0,\alpha}(\overline{\Omega}_m)}=0.$$
Then for every $m\in\mathbb{N}$ large enough, there exists a $J$-holomorphic map $u_m\colon \Omega_{m}\to\mathbb{R}^{2n}$ such that $$\lim_{m\to\infty}\left\|u_m-\varphi_{m}\right\|_{\mathcal{C}^{1,\alpha}(\overline{\Omega}_m)}=0.$$  
\end{proposition}

\begin{proof}
We need to verify that the constants $c_0>0$ and $\delta>0$ in Theorem 2 can be chosen independently of the sets $\Omega_m$. For this, we need two bounded linear extension operators $E_k\colon \mathcal{C}^{j,\alpha}(\overline{\mathbb{D}})\to \mathcal{C}^{j,\alpha}(\overline{\Omega}_m)$, $j=0,1$ (see e.g. \cite[Theorem 4, p.177]{Stein}). The remarkable fact is that their norms can be bounded by a constant $K_j\geq 1$ that is independent of the sets $\Omega_m$.
 
For points in $\Omega_m$ and $h\in\mathcal{C}^{1,\alpha}(\overline{\mathbb{D}})$ we have 
$d_{\varphi_m}\mathcal{F}_m(h_m)=d_{\varphi}\mathcal{F}(h).$ Hence, one can construct a bounded right inverse $Q_m$ of $d_{\varphi_m}\mathcal{F}_m$ by using the right inverse $Q_{\varphi}$ of $d_{\varphi}\mathcal{F}$. Indeed, given $g_m\in \mathcal{C}^{0,\alpha}(\overline{\Omega}_m)$, we take its extension $g=E_0(g_{m})$ and proclaim $h_m=Q_{\varphi_m}g_m$ to be the restriction of $h=Q_{\varphi}g$ to $\Omega_m$. If $\left\|Q_{\varphi}\right\|<c_0$ and $\left\|\varphi\right\|_{\mathcal{C}^{0,\alpha}(\overline{\mathbb{D}})}<c_0$, the following estimate is valid 
$$\left\|Q_{\varphi_m}g_m\right\|_{\mathcal{C}^{0,\alpha}(\overline{\Omega}_m)}\leq \left\|Q_{\varphi}g\right\|_{\mathcal{C}^{0,\alpha}(\overline{\mathbb{D}})}\leq 
%\left\|Q_{\varphi}\right\|\cdot \left\|g\right\|_{\mathcal{C}^{0,\alpha}(\overline{\mathbb{D}})}
c_0K_0\cdot \left\|g_m\right\|_{\mathcal{C}^{0,\alpha}(\overline{\Omega}_m)}.$$
Moreover, $\left\|\varphi_m\right\|_{\mathcal{C}^{1,\alpha}(\overline{\Omega}_m)}< c_0 \leq c_0 K_0.$  

Further, given $m\in\mathbb{N}$ let $\varphi_m,\tilde{\varphi}_m\in \mathcal{C}^{1,\alpha}(\overline{\Omega}_m)$ be such that 
$$\left\|\tilde{\varphi}_m-\varphi_m\right\|_{\mathcal{C}^{1,\alpha}(\overline{\Omega}_m)}<\frac{1}{K_1}.$$
For $\varphi=E_1(\varphi_{m})$ and $\tilde{\varphi}=E_1(\tilde{\varphi}_{m})$ we have 
$\left\|\tilde{\varphi}-\varphi\right\|_{\mathcal{C}^{1,\alpha}(\overline{\mathbb{D}})}<1.$
Hence as in $(\ref{eqlip})$ we can conclude that $$\left\|d_{\tilde{\varphi}_m}\mathcal{F}_m(h)-d_{\varphi_m}\mathcal{F}_m(h)\right\|_{\mathcal{C}^{0,\alpha}(\overline{\Omega}_m)}
%\leq\left\|d_{\tilde{\varphi}}\mathcal{F}(h)-d_{\varphi}\mathcal{F}(h)\right\|_{\mathcal{C}^{0,\alpha}(\overline{\Delta})}
%\leq c\|\tilde{\varphi}-\varphi\|_{\mathcal{C}^{1,\alpha}(\overline{\Delta})}\|h\|_{\mathcal{C}^{1,\alpha}(\overline{\Delta})}
\leq cK_1^2\|\tilde{\varphi}_m-\varphi_m\|_{\mathcal{C}^{1,\alpha}(\overline{\Omega}_m)}\|h\|_{\mathcal{C}^{1,\alpha}(\overline{\Omega}_m)}.$$
We now get the desired result from Theorem 2 by setting 
$$\eta=\min\left\{\frac{1}{K_1},\frac{1}{2cc_0K_0K_1^2}\right\}\;\;\textrm{and}\;\;\displaystyle \delta=\frac{\eta}{4c_0K_0}.$$ 
\end{proof}
We now prove a local approximation statement that will be used in the proof of the main theorem. We mimic the proof of \cite[Lemma 3.5]{CHAK}.
\begin{theorem}\label{arc approximation}
Let $J\in\mathcal{J}$. Given $\epsilon>0$, a smoothly embedded arc $\Gamma\subset\mathbb{C}$ and a $\mathcal{C}^2$-map $\varphi\colon \Gamma\to\mathbb{R}^{2n}$, there exists a neighborhood $U$ of $\Gamma$ and a $J$-holomorphic map $u\colon U\to \mathbb{R}^{2n}$ such that
$\left\|u-\varphi\right\|_{\mathcal{C}^{1,\alpha}(\Gamma)}<\epsilon.$
\end{theorem}

\begin{proof}
Without loss of generality we can assume that $\Gamma\subset\mathbb{D}\cap\mathbb{R}$. By (3) the $J$-holomorphicity condition equals  $u_y=J(u)u_x$. Hence, we may extend $\varphi$ to a function that is quadratic in $y$ and whose $\bar{\partial}_J$-derivative vanishes up to the first order along $\Gamma$. In particular, for 
$$\Omega_m=\left\{z\in\mathbb{C}:\; \textrm{dist}(z,\Gamma)<\frac{1}{m}\right\}$$
there exist $m_0>0$ and $C_\alpha>0$ such that $\Omega_m\subset\mathbb{D}$ and  
$$\left\|(\varphi_m)_{\bar{\zeta}} +A(\varphi_m)\overline{(\varphi_m)_\zeta}\right\|_{\mathcal{C}^{0,\alpha}(\overline{\Omega}_m)}< \frac{C_\alpha}{m}$$
for every $m\geq m_0$. The rest follows from the Proposition 4 above.
\end{proof}
% adopt this proof slightly in order to obtain the following corolarry.
%\begin{coro}\label{Runge coro}
%Let $M$ be a connected manifold satisfying the propositions of Theorem $2$. Let $K\subset\partial\D$ be a compact set and let the map $\varphi\colon U\to M$  be $J$-holomorphic on a neighborhood of $K\cup\left\{0\right\}$. Given $\epsilon>0$ there exist a $J$-holomorphic disc $u$ such that $\left\|u-\varphi\right\|_{\mathcal{C}^0(K)}<\epsilon$ and $u(0)=\varphi(0)$. 
%\end{coro}

\section{The Runge-type approximation}

As mentioned in the introduction, our result is an application of the following Runge-type theorem provided by A. Gournay \cite{Gournay}. 
\begin{theorem}\label{Runge}
Let $M$ be a smooth compact manifold equipped with a regular almost complex structure and admitting a doubly tangent property.
Suppose we are given $\epsilon>0$, a compact Riemann surface $\Sigma$, an open set $U\subset\Sigma$, a $J$-holomorphic map $\varphi\colon U\to M$ and a compact set $K\subset U$. Then, provided that there is a $\mathcal{C}^0$ extension of $\varphi$ to $\Sigma$, there exists a $J$-holomorphic map $u\colon \Sigma\to M$ such that $\left\|u-\varphi\right\|_{\mathcal{C}^0(K)}<\epsilon$ .\end{theorem}  
\noindent We include below a brief discussion on the proof in order to explain why the statement can be applied in the present case.

Provided that there are no topological obstructions, one can define a new map extending the initial data $\varphi|_K$ to $\Sigma$ in a $\mathcal{C}^{\infty}$-fashion. Let us denote it by $\varphi$ again. Of course such a map needs not to be holomorphic on $\Sigma\setminus K$ and one can express this locally. Fix $\zeta_0\in\Sigma\setminus K$ and the following two charts: a chart $\psi$ on $\Omega_0\subset\Sigma$ with $\psi(\zeta_0)=0$ and a chart $\phi$ on $M$ taking $q=\varphi(\zeta_0)\in M$ to $0\in\mathbb{R}^{2n}$ and satisfying $\phi^*(J)(0)=J_{st}$. There exist $a,b\in\mathbb{R}^{2n}$ such that
$$\psi\circ \varphi\circ \phi^{-1} (z)= az+b\bar{z}+O(|z|^2),$$
where $b\neq 0$ is equivalent to $\bar{\partial}_J\varphi\neq 0$. 

First of the two key assumptions in the method of A. Gournay is that the manifold enjoys the \textit{doubly tangent property}. That is, for almost every $q\in M$ and almost every pair $a,b\in\mathbb{R}^{2n}$, there exists a $J$-holomorphic sphere $H_{a,b}^r\colon \mathbb{C}P^1\to M$ whose local (Laurent) expansion equals to 
$$\psi\circ H_{a,b}^r\circ \phi^{-1}(z)=az+br^2/z+O(r^{1+\epsilon}).$$ 
Here, $\epsilon>0$ and $r>0$ are such that for $\frac{r}{(1+r^{\epsilon})}<|z|<r(1+r^{\epsilon})$ we have $\psi^{-1}(z)\in \Omega_0.$ Hence, using an appropriate cut-off function, the map $\varphi$ may be replaced by $H_{a,b}^r$ in the vicinity of $\zeta_0$. Moreover, since $H_{a,b}^r$ is holomorphic and almost agrees with $\varphi$ on $|z|=r$, such a surgery diminishes the 'size' of the $\bar{\partial}_J$-derivative. The author calls such a procedure \textit{grafting} and he repeats it finitely many times until reaching the desired bounds (see \cite[\S 3.1]{Gournay}). 

Once an approximate solution is constructed (we denote it by $\varphi\colon \Sigma\to M$ again), the  $\bar{\partial}_J$-equation can be solved similarly as in our \S 1. Let us briefly explain this. The non-linear $\bar{\partial}_J$-operator, now defined globally, may be linearized at a compact curve $u$ so that the corresponding Fredholm operator $D_{u}$ maps from $\mathcal{C}^{\infty}(\Sigma,\varphi^*TM)$ to $\mathcal{C}^{\infty}(\Sigma,\Lambda^{0,1}\varphi^*TM)$ (see \cite[(3.1.4.)]{McDuff}). The notion of regularity refers to its surjectivity. In particular, for us the structure $J$ is \textit{regular} when $D_u$ is onto for every $J$-holomorphic sphere. The idea is to find a bounded right inverse for $D_\varphi$. That is, given $\eta\in\mathcal{C}^{\infty}(\Sigma,\Lambda^{0,1}\varphi^*TM)$ we seek solution $\xi\in \mathcal{C}^{\infty}(\Sigma,\varphi^*TM)$ of $D_{\varphi}\xi=\eta$ with bounds. 

In the cited paper the latter is obtained in two steps. Firstly, it is shown in \cite[\S 3.2]{Gournay} that it suffices to solve local equations $D_{\varphi_j}\xi_j=\eta_j$, where $\left\{(\varphi_j, \xi_j, \eta_j)\right\}_{j\geq 0}$ stands for slightly perturbed data $(\varphi, \xi, \eta)$ restricted either to the original surface $\Sigma_0=\Sigma$ or to one of the finitely many grafts $\Sigma_j=\mathbb{C}P^1$, $j>0$. Secondly, it is proved in \cite[\S 3.3]{Gournay} that though the local equations for $j>0$ interact with the one for $j=0$ the iteration starting at $\xi_0=0$ is indeed contractible. Here the regularity of the structure is crucial since it ensures that the local equation is always solvable along the grafts. In contrast, the inversion of the linear equation for $j=0$, is very subtle \cite[\S 3.3]{Gournay}.  

Finally, it is worth mentioning that the norms in question are not the ones associated with Sobolev or H\"older spaces. The reason lies in the fact that each graft increases the $L^p$ norm of $d\varphi$ by a quantity that is a priori significant. Furthermore, the number of surgeries is not bounded in general. Hence, the local Lipschitz constant grows with $\partial\varphi$ when $D_{\varphi}$ is treated as a map from $W^{1,p}(\Sigma,\Lambda^{0,1}\varphi^*TM)$ to $L^{p}(\Sigma,\varphi^*TM),$ $p>2$. This makes it impossible to use the Implicit function theorem. Hence a certain sup-norm introduced by C. Taubes \cite{Taubes} is used (see also \cite[\S 4]{Donaldson}).

We now state the corollary that will be used in the proof of Theorem $1$. 
\begin{coro}\label{Runge coro}
Let $(M,J)$ be as in Theorem $6$. Let $K\subset\partial\D$ be a compact set and let the map $\varphi$ be continuous near $\overline{\mathbb{D}}$ and $J$-holomorphic near $K\cup\left\{0\right\}$. Given $\epsilon>0$, there exists $u\in \cO(\clD,M,\varphi(0))$ such that $\left\|u-\varphi\right\|_{\mathcal{C}^0(K)}<\epsilon$. 
\end{coro}
\begin{proof}
First note that the map $\varphi$ can be continuously extended to $\Sigma=\mathbb{C}P^1$. Hence, the Runge-type theorem guarantees existence of a $J$-holomorphic map approximating $\varphi$ on $K\cup\left\{0\right\}$. It remains to explain why the above proof can be adopted slightly in order to obtain $u(0)=\varphi(0)$. 

The simplest way to do this is by adding an 'un-necessary graft' at the center. That is, we start by replacing the map $\varphi$ with an appropriate graft $H_{a,b}^r$ near $0\in\mathbb{D}$. Since $\bar{\partial}_J\varphi(0)=0$ we have $b=0$ and $\varphi(0)=H_{a,b}^r(0)$ here. We index this graft with $j=1$ and then proceed with the usual grafting procedure for $j>1$. Moreover, we add a point-wise restriction $\xi(0)=0$ each time when solving the local linear equation $D_{\varphi_1}\xi_1=\eta_1$. Since $J$ is regular this does not object the surjectivity of $D_{\varphi_1}$. Indeed, check \cite[\S 3.4]{McDuff}. Hence, all the key estimates remain fulfilled. In particular, \cite[Corollary 2.5.5.]{Gournay} can be used in the iterative scheme from \cite[\S 3.3]{Gournay}. 
\end{proof}

\section{Proof of Theorem 1}
The direct construction of a Poletsky disc follows from Theorem \ref{arc approximation} and Corollary \ref{Runge coro}. 
Indeed, as in \cite{Rosay3}, we prove the following stronger statement.  
\begin{theorem}
Let $(M,J)$ be as in Theorem 1 and equipped with some Riemannian metric. Given a point $p\in M$, a positive constant $\epsilon>0$ and a $\mathcal{C}^2$-map $\lambda\colon \partial\mathbb{D}\to M$, there exist a disc $u\in\cO(\clD,M,p)$ and a set $E\subset[0,2\pi)$ such that $|E|<\epsilon$ and $\textrm{dist}(u(e^{it}),\lambda(e^{it}))<\epsilon$ for $t\in [0,2\pi)\setminus E$.  
\end{theorem}
\begin{proof}
As pointed out above the direct method consists of two steps. We first make a piece-wise holomorphic approximation of $\lambda$. Then we use the Runge-type theorem to extend this map to the whole disc. The second step can be understood as adding finitely many poles (grafts).   

Fix $e^{it}\in\partial\D$. Let $\phi_t\colon V_t\to \mathbb{B}$ be a local chart mapping a neighborhood of $\lambda(e^{it})$ into a neighborhood of the origin in $\mathbb{R}^{2n}$ and satisfying $\phi_t^*(J)\in\mathcal{J}$. We define $\Gamma'_t\subset\partial\D$ to be the largest connected subarc including $e^{it}$ and satisfying  $\lambda(\Gamma'_t)\subset\lambda(\partial\D)\cap V_t$. By compactness, there are points $t_1,\ldots,t_k\in[0,2\pi)$ such that the the union $\cup_{j=1}^k \Gamma'_{t_j}$ covers the whole $\partial\D$. Moreover, we can choose smaller pairwise disjoint subarcs $\Gamma_{t_j}\subset\subset\Gamma_{t_j}'$ satisfying 
$$|\partial\D\setminus\cup_{j=1}^k \Gamma_{t_j}|<\epsilon.$$ 

By Theorem \ref{arc approximation} there exist $J$-holomorphic maps $u_j\colon U_j\to V_{z_j}$ that are defined on pairwise disjoint neighborhoods and $\mathcal{C}^{1,\alpha}$-close to $\phi_j\circ\lambda$ on $\Gamma_{t_j}.$ Moreover, by the classical Nijenhuis-Woolf theorem \cite{NW} there exists a small $J$-holomorphic disc $u_0$ centered at $p$. Since $M$ is connected we can join these pieces into a continuous map $\varphi$ defined on a neighborhood of $\overline{\mathbb{D}}$ and satisfying $\varphi(0)=u_0(0)=p$. The rest follows from the Corollary \ref{Runge coro} applied to the compact set $K=\cup_{j=1}^k \Gamma_{t_j}$.
\end{proof}
%Let $u\colon \Delta\to M$ and $v\in \C$. We define the operators
%$$\partial_J u(v)=\frac{1}{2}\left(du(v)-J(u)du(i v)\right),\;\;
%\overline{\partial}_J u(v)=\frac{1}{2}\left(du(v)+J(u)du(i v)\right).$$
%Note that $u$ is a $J$-holomorphic disc if and only if $\overline{\partial}_J u= 0.$ We denote by $u\colon\overline{\Delta}\to M$ the discs that are $J$-holomorphic on some neighborhood of $\overline{\Delta}.$

%Finally, recall that an upper semi-continuous function $f$ defined on an open set $V$ in $(M,J)$ is 
%{\it $J$-plurisubharmonic} if 
%$\displaystyle f \circ u$ is subharmonic for any $J$-holomorphic disc $u: \Delta \to V$. We denote by $\Psh(V)$ the set of $J$-plurisubharmonic
%functions on $V$.  

\thanks{The research of the author was supported in part by the research program P1-0291 and the grant J1-7256 from ARRS, Republic of Slovenia. A large part of the result was created during his stay at the University of Oslo, Spring 2017. He wants to thank prof. Erlend F. Wold for his hospitality. He also wants to thank prof. Barbara Drinovec-Drnov\v sek for her useful remarks on an earlier version of the paper.}

\end{document}